\documentclass[12pt]{article}
\usepackage{amsfonts,amssymb,curves}
\newtheorem{theorem}{Theorem}[section]
\newtheorem{prop}[theorem]{Proposition}

\newtheorem{unnumber}{}

\newtheorem{prepf}[unnumber]{Proof}
\newenvironment{proof}{\prepf\rm}{\endprepf}
\newcommand{\qed}{\hfill$\Box$}
\newtheorem{preex}{Example}
\newenvironment{example}{\preex\rm}{\endpreex}

\newcommand{\Aut}{\mathop{\mathrm{Aut}}\nolimits}
\newcommand{\Sym}{\mathop{\mathrm{Sym}}\nolimits}

\begin{document}

\title{The diagonal graph}
\author{R. A. Bailey and Peter J. Cameron\\
School of Mathematics and Statistics, University of St Andrews,\\
North Haugh, St Andrews, Fife, KY16 9SS, U.K.}
\date{}
\maketitle

\begin{abstract}
According to the O'Nan--Scott Theorem, a finite primitive permutation group
either preserves a structure of one of three types (affine space, Cartesian
lattice, or diagonal semilattice), or is almost simple. However, diagonal
groups are a much larger class than those occurring in this theorem. For any
positive integer $m$ and group $G$ (finite or infinite), there is a diagonal
semilattice, a sub-semilattice of the lattice of partitions of a set $\Omega$,
whose automorphism group is the corresponding diagonal group. Moreover, there
is a graph (the diagonal graph), bearing much the same relation to the
diagonal semilattice and group as the Hamming graph does to the Cartesian
lattice and the wreath product of symmetric groups.

Our purpose here, after a brief introduction to this semilattice and graph, is
to establish some properties of this graph. The diagonal graph 
$\Gamma_D(G,m)$ is a Cayley graph for the group~$G^m$, and so is
vertex-transitive. We establish its clique number in general and its 
chromatic number in most cases, with a conjecture about the chromatic number
in the remaining cases. We compute the spectrum of the adjacency matrix of
the graph, using a calculation of the M\"obius function of the diagonal
semilattice. We also compute some other graph parameters and symmetry
properties of the graph.

We believe that this family of graphs will play a significant role in
algebraic graph theory.
\end{abstract}

\section{Introduction}

Let $m$~be a positive integer and $G$~a group, finite or infinite. The
\emph{diagonal group} $D(G,m)$ is the group of permutations of~$G^m$ generated
by the following transformations:
\begin{itemize}\itemsep0pt
\item $G^m$, acting by right multiplication;
\item $G$, acting diagonally by left multiplication (that is, the element
$x$~of~$G$ induces the map
$(g_1,\ldots,g_m)\mapsto(x^{-1}g_1,\ldots,x^{-1}g_m)$);
\item automorphisms of~$G$, acting coordinatewise;
\item the symmetric group $\Sym(m)$, acting by permuting the coordinates;
\item the map $(g_1,g_2,\ldots,g_m)\mapsto(g_1^{-1},g_1^{-1}g_2,\ldots,
g_1^{-1}g_m)$.
\end{itemize}
It is a transitive permutation group of order $|G|^m\cdot|\Aut(G)|\cdot(m+1)!$.

Diagonal groups play an important role in the O'Nan--Scott theorem describing
finite primitive permutation groups, given in~\cite{scott}.
One of the types which occur consists of the \emph{simple diagonal groups},
subgroups of $D(G,m)$ where $G$~is a non-abelian finite simple group.

In the case that $m=1$, the diagonal group $D(G,1)$ is the group of permutations
of $G$ generated by left and right translation, automorphisms, and inversion.
In this case there is no geometry associated with the group, although it does
preserve a fusion of the \emph{conjugacy class association scheme} of~$G$;
we do not consider this case further, but usually assume that $m\geqslant2$ in
what follows.

We now outline the contents of the paper. We give a brief introduction to the
lattice of partitions of a set in Section~\ref{s:partitions}; then we define
the semilattice $\mathcal{D}(G,m)$ associated with the diagonal group
$D(G,m)$, and the corresponding \emph{diagonal graph} $\Gamma_D(G,m)$. We
identify this graph with the Latin square graph of the Cayley table of $G$ if
$m=2$, and with the folded cube if $G=C_2$, and give an example satisfying
neither of these conditions.

Sections~\ref{s:mobius} and~\ref{s:spectrum} calculate the M\"obius function
of the diagonal semilattice and use this to find the spectrum of the diagonal
graph. Section~\ref{s:chromatic} states a theorem and conjecture on the
chromatic number of the diagonal graph. The final section gives some further
properties of the diagonal graph: it is a Cayley graph, but not in general
either edge-transitive or distance-regular; we give a formula for its diameter.

\section{The partial order on  partitions}
\label{s:partitions}

In this section, we briefly remind readers about the partial order on
partitions.

Let $P$ and $Q$ be two partitions of a given set~$\Omega$.  Then $P$ is
defined to be \textit{finer} than $Q$, written $P\prec Q$, if every
part of $P$ is contained in a single part of $Q$ but $P\ne Q$.
Write $P\preccurlyeq Q$ to mean that $P\prec Q$ or $P=Q$. Then
$\preccurlyeq$ is a partial order.

If $\left| \Omega \right| \geqslant 2$ then there are two trivial partitions of
$\Omega$.  The universal partition $U$ has a single part; at the other extreme,
the parts of the equality partition $E$ are all singletons.  Thus
$E \preccurlyeq P \preccurlyeq U$ for every partition~$P$ of $\Omega$.

The \textit{join} $P\wedge Q$ of $P$ and $Q$ is the partition whose parts are
all non-empty intersections of a part of $P$ with a part of~$Q$.  Thus
$P \wedge Q \preccurlyeq P$, $P \wedge Q \preccurlyeq Q$, and, if $R$ is any
partition of $\Omega$ satisfying $R\preccurlyeq P$ and $R \preccurlyeq Q$,
then $R \preccurlyeq P \wedge Q$.

The dual notion is the \textit{meet} $P \vee Q$ of $P$ and~$Q$.  Draw a graph
with vertext-set~$\Omega$, joining two distinct vertices if they lie in the
same part of~$P$ or the same part of~$Q$.  Then the parts of $P\vee Q$ are
the connected components of this graph.  It follows that $P \preccurlyeq
P \vee Q$, $Q \preccurlyeq P \vee Q$, and, if $R$ is any partition of $\Omega$
satisfying $P\preccurlyeq R$ and $Q \preccurlyeq R$, then
$P \vee Q \preccurlyeq R$.

A set of partitions which is closed under both join and meet is called a
\textit{lattice}. We have to consider sets which are closed under
join but not under meet; such a set called a \textit{join semi\-lattice}. See
the next section for how our examples arise.

Partial orders are often represented by \textit{Hasse diagrams}.  For a
collection of partitions of a set, there is one dot for each partition.
If $P \prec Q$ then the dot for $P$ is lower in the diagram than the dot
for~$Q$.  If, in addition, there is no partition~$R$ in the collection with
$P\prec R \prec Q$, then there is a line joining the dots for $P$ and $Q$.
We give two examples in the next section.

Given any partially ordered set~$\mathcal{S}$, its \textit{zeta function}
is the $\mathcal{S} \times \mathcal{S}$ matrix defined by $\zeta(P,Q)=1$
if $P \preccurlyeq Q$ and $\zeta(P,Q)=0$ otherwise.  This matrix is
invertible, and it inverse is called the \textit{M\"obius function} $\mu$
for $\mathcal{S}$. See \cite{rota}.

\section{Diagonal semilattice and diagonal graph}
\label{s:diagonal}
We begin by defining some (semi)lattices of partitions of a set $\Omega$.

A \emph{Latin square} is usually defined as a square array of size
$q\times q$ with entries from an alphabet of size $q$. We can interpret this
as a lattice of partitions of the cells of the array containing five
partitions $\{E,R,C,L,U\}$, where $E$ is the partition into singletons,
$U$ is the partition with a single part, and $R$, $C$ and $L$ are the partitions
into rows, columns, and letters. We have $R\wedge C = R\wedge L = C \wedge
L = E$ and $R \vee C = R \vee L = C \vee L = U$. The
set $\{E,R,C,L,U\}$ of partitions
determines the Latin square up to paratopism. The Hasse diagram is shown in
Figure~\ref{fig:LS}.

\begin{figure}
  \begin{center}
\setlength{\unitlength}{1mm}
\begin{picture}(30,30)
\multiput(15,5)(0,20){2}{\circle*{2}}
\multiput(5,15)(10,0){3}{\circle*{2}}
\multiput(5,15)(10,10){2}{\line(1,-1){10}}
\multiput(5,15)(10,-10){2}{\line(1,1){10}}
\put(15,5){\line(0,1){20}}
\put(10,3){$E$}
\put(10,13){$C$}
\put(10,23){$U$}
\put(0,13){$R$}
\put(27,13){$L$}
\end{picture}
\end{center}
  \caption{The Hasse diagram for a Latin square}
  \label{fig:LS}
  \end{figure}
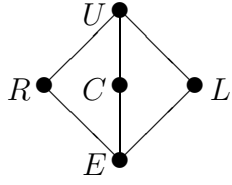

If $\Omega=A^m$ (the set of $m$-tuples over an alphabet $A$),
we define the $m$-dimensional \emph{Cartesian lattice} to be the lattice
consisting of the partitions $Q_I$ for all $I\subseteq\{1,\ldots,m\}$, where
$Q_I$ corresponds to the equivalence relation in which two $m$-tuples are
equivalent if they agree in all coordinates outside the subset $I$. It is
isomorphic to the Boolean lattice of subsets of the set $\{1,\ldots,m\}$ by
the map $I\mapsto Q_I$ for all $I\subseteq\{1,\ldots,m\}$.

Let $G$ be a group and $m$ a positive integer. In the first section we
defined the diagonal group $D(G,m)$ as a permutation group on the set
$\Omega=G^m$. Now we construct on $\Omega$ a semilattice of partitions, and a
graph, which admit the diagonal group as a group of automorphisms.

For $1\leqslant i\leqslant m$, let $Q_i$ be the partition of $\Omega$ defined
by the equivalence relation $\equiv_i$, where
\[(g_1,\ldots,g_m)\equiv_i(h_1,\ldots,h_m)\Leftrightarrow(\forall j\ne i)
(g_j=h_j);\]
this is the orbit partition of the $j$th coordinate subgroup of $G^m$. Also,
let $Q_0$ be the orbit partition of the group $G$ acting on $\Omega$ by
left multiplication, that is, corresponding to the equivalence relation
$\equiv_0$ defined by
\[(g_1,\ldots,g_m)\equiv_0(h_1,\ldots,h_m)\Leftrightarrow(\exists x\in G)
(\forall i)(h_i=xg_i).\]
The \emph{diagonal semilattice} $\mathcal{D}(G,m)$ is the join-semilattice
generated by the partitions $Q_0,Q_1,\ldots,Q_m$. Figure~\ref{fig:dH}
shows the Hasse diagram of $\mathcal{D}(G,3)$.

\begin{figure}
  \newcommand{\blobb}{\circle*{1}}
    \begin{center}
    \setlength{\unitlength}{2mm}
    \begin{picture}(60,40)
      \put(55,15){\line(0,1){10}}
      \put(55,15){\line(-1,1){10}}
      \put(55,15){\line(-3,1){30}}
      \put(15,15){\line(-1,1){10}}
      \put(15,15){\line(1,1){10}}
      \put(15,15){\line(0,1){10}}
      \put(35,15){\line(1,1){10}}
      \put(35,15){\line(0,1){10}}
      \put(35,15){\line(-3,1){30}}
      \put(45,15){\line(-1,1){10}}
      \put(45,15){\line(-3,1){30}}
      \put(45,15){\line(1,1){10}}
      \put(30,5){\line(1,2){5}}
      \put(30,5){\line(-3,2){15}}
      \put(30,5){\line(3,2){15}}
      \curve(30,5,55,15)
      \put(30,35){\line(-1,-2){5}}
      \put(30,35){\line(1,-2){5}}
      \put(30,35){\line(-3,-2){15}}
      \put(30,35){\line(3,-2){15}}
      \curve(30,35,5,25)
      \curve(30,35,55,25)
      \put(35,15){\blobb}
      \put(36,15){\makebox(0,0)[l]{$Q_1$}}
      \put(15,15){\blobb}
      \put(14,15){\makebox(0,0)[r]{$Q_0$}}
      \put(45,15){\blobb}
      \put(56,15){\makebox(0,0)[l]{$Q_3$}}
      \put(55,15){\blobb}
      \put(46,15){\makebox(0,0)[l]{$Q_2$}}
      \put(30,5){\blobb}
      \put(30,3){\makebox(0,0){$E$}}
\put(5,25){\blobb}
\put(2,25){\makebox(0,0){$Q_{\{0,1\}}$}}
\put(15,25){\blobb}
\put(12,25){\makebox(0,0){$Q_{\{0,2\}}$}}
\put(25,25){\blobb}
\put(22,25){\makebox(0,0){$Q_{\{0,3\}}$}}
\put(35,25){\blobb}
\put(36,25){\makebox(0,0)[l]{$Q_{\{1,2\}}$}}
\put(45,25){\blobb}
\put(46,25){\makebox(0,0)[l]{$Q_{\{1,3\}}$}}
\put(55,25){\blobb}
\put(56,25){\makebox(0,0)[l]{$Q_{\{2,3\}}$}}
\put(30,35){\blobb}
\put(30,37){\makebox(0,0){$U$}}
    \end{picture}
    \end{center}
    \caption{The Hasse diagram for a diagonal semi-lattice with dimension~$3$}
    \label{fig:dH}
\end{figure}
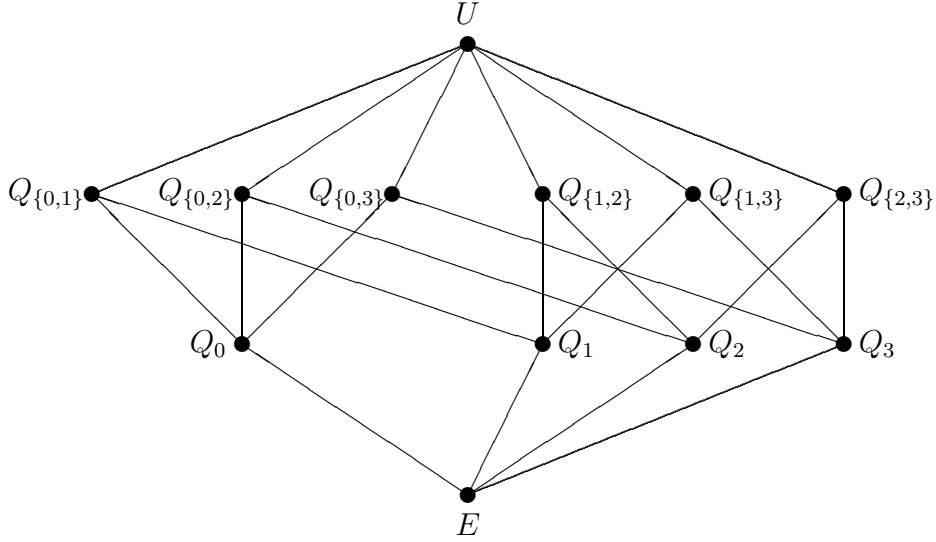

Note that, in a Latin square (regarded as a partition lattice) or a diagonal
semilattice, if there are $m+1$ minimal non-trivial partitions, then any $m$
of them generate a Cartesian lattice (in the Latin square case, a square grid).
Also, the diagonal semilattice $\mathcal{D}(G,2)$ is isomorphic to the lattice
defined by a Latin square which is the Cayley table of $G$.

Now we can state the main theorems of~\cite{diag}.

\begin{theorem}\label{t:main}
Let $m\ge2$, and let $\Omega$ be a set and $Q_0,Q_1,\ldots,Q_m$ a set of
partitions of $\Omega$. Suppose that any $m$ of these partitions generate a
lattice that is isomorphic to a Cartesian lattice.
\begin{enumerate}\itemsep0pt
\item If $m=2$, then the three partitions together with the trivial partitions
$E$ and $U$ form a Latin square, unique up to paratopism.
\item If $m\ge3$, then there is a group $G$, unique up to isomorphism, such
that the $m+1$ partitions generate the diagonal semilattice $\mathcal{D}(G,m)$.
\end{enumerate}
\end{theorem}

\begin{theorem}
Let $G$ be a group and $m\ge2$. Then the automorphism group of the diagonal
semilattice $\mathcal{D}(G,m)$ is the diagonal group $D(G,m)$ defined earlier.
\end{theorem}

Associated with a semi\-lattice satisfying the hypotheses of
Theorem~\ref{t:main}
is a
graph, defined as follows: the vertex set is $\Omega=G^m$; two vertices are
adjacent if they are contained in a part of one of the partitions 
$Q_0,\ldots,Q_m$. If $m=2$ and $Q_0,Q_1,Q_2$ are the rows, columns and letters
of a Latin square $\Lambda$, then this graph is precisely the \emph{Latin
  square graph} associated with $\Lambda$.
If $Q_0,\ldots,Q_m$ are the minimal
non-trivial partitions in a diagonal semilattice $\mathcal{D}(G,m)$, then
the graph is the \emph{diagonal graph}, which we denote by $\Gamma_D(G,m)$.

For later use, we note that, if $m=1$, the diagonal semilattice has just two
partitions $E$ and $U$ (indeed $Q_0=Q_1=U$), and the diagonal graph is a
complete graph of order~$|G|$.

\begin{theorem}
If $m>2$, or if $m=2$ and $|G|>4$, then the automorphism group of the diagonal
graph $\Gamma_D(G,m)$ is the diagonal group $D(G,m)$.
\end{theorem}

We conclude this section with further examples.

\begin{example}
Take $G=C_2$, the cyclic
group of order $2$. Then the vertex set of $\Gamma_D(C_2,m)$ is the set
of all binary $m$-tuples, and the partitions $Q_1,\ldots,Q_m$ correspond to
the usual parallel classes of edges of the $m$-dimensional cube, while $Q_0$
consists of the antipodal pairs of vertices. So $\Gamma_D(m,C_2)$ is the
$m$-cube with antipodal vertices joined. This is the \emph{folded cube}
(see~\cite[p.~264]{bcn}).
\end{example}

\begin{example}
Let $m=3$ and $G=C_3$.  Let $x$ be a generator of $G$, and put $a = (x,1,1)$,
$b= (1,x,1)$ and $c=(1,1,x)$.  Then $\Omega = G^3$ and the parts of
$Q_0$, $Q_1$, $Q_2$ and $Q_3$ are the (right) cosets of the subgroups
$\langle abc \rangle$, $\langle a \rangle$, $\langle b \rangle$ and
$\langle c \rangle$ respectively.
The vertices of the diagonal graph $\Gamma_D(G,3)$ are the elements of $G^3$.
Let $g$ and $h$ be two such vertices.  These are joined by an edge if and
only if $gh^{-1}$ is one of $a$, $a^2$, $b$, $b^2$, $c$, $c^2$, $abc$ and
$a^2b^2c^2$.  Thus $\Gamma_D(G,3)$ is regular with valency~$8$.

Vertices $1$ and $ab$ are at distance two.  Their common neighbours are
$a$, $b$, $c$ and $abc$.  Vertices $1$ and $a^2b$ are also at distance two,
but their common neighbours are only $a^2$ and $b$.  Therefore this
diagonal graph is not distance-regular in the sense of \cite{bcn}.
\end{example}

\section{The M\"obius function of the diagonal semilattice}
\label{s:mobius}

Let $m$ be a natural number greater than $1$.
Let $Q_0,\ldots,Q_m$ be $m+1$ partitions of a set $\Omega$ with the property
that any $m$ of them are the minimal elements of a Cartesian lattice of
partitions of $\Omega$. Let $q$ be the (constant) size of a part of one of
the~$Q_i$. (The constancy of $q$ is shown in~\cite{diag}.)

By Theorem~\ref{t:main}, either $m=2$ and
$Q_0,Q_1,Q_2$ are the row, column and letter partitions of the set of cells
of a Latin square (unique up to paratopism), or $m\geqslant3$ and
$Q_0,\ldots,Q_m$ are the minimal partitions of the diagonal semilattice
$\mathcal{D}(G,m)$ of dimension $m$ over a group $G$ (unique up to isomorphism),
with automorphism group the diagonal group $D(G,m)$.

In the case $m=2$, the graph with vertex set $\Omega$, with edges
joining two elements of $\Omega$ if and only if they are contained
in a part of one of the three partitions, is the
Latin square graph associated with the Latin square.
 Latin square graphs form a prolific family
of \emph{strongly regular graphs}: they give rise to more than exponentially
many such graphs with the same parameters on a given (square) number of
vertices. The spectrum of a Latin square graph is well known. We
extend this to higher-dimensional diagonal graphs. First, we compute
the \emph{M\"obius function}~\cite{rota} of the diagonal semilattice.
For each partition~$S$ in the semi\-lattice,
let $\rho(S)$ be its \emph{rank}, which is the
number of steps in a maximal chain from $E$ to~$S$, so that $\rho(U)=m$.

\begin{theorem}
\label{th:Mob}
  The M\"obius function of $\mathcal{D}(G,m)$ is given by
\[\mu(S,T)=\cases{(-1)^{\rho(T)-\rho(S)} & if $T\ne U$,\cr
  (-1)^{\rho(U)-\rho(S)}(\rho(U)-\rho(S)) = (-1)^{m-\rho(S)}(m-\rho(S)) &
  if $T=U$,\cr}\]
for $S\preccurlyeq T$.
\end{theorem}

\begin{proof} Every interval $[E,T]$ with $T\ne U$ is a isomorphic to a
  Boolean lattice of rank equal to $\rho(T)$.
So the values of $\mu(S,T)$ for
$T\ne U$ are equal to those in the Boolean lattice, and are as claimed.

As shown in \cite{diag}, each of the remaining intervals $[S,U]$ is a
diagonal semilattice $\mathcal{D}(G,m-\rho(S))$.
Therefore we can obtain $\mu(S,U)$ by replacing $m$ by $m-\rho(S)$ in
the formula for $\mu(E,U)$.

Thus it suffices to calculate $\mu(E,U)$. This we can now do as
follows:
\begin{eqnarray*}
\mu(E,U) = -\sum_{E\preccurlyeq S\prec U}\mu(E,S) \\
= -\sum_{i=0}^{m-1}(-1)^i{m+1\choose i},
\end{eqnarray*}
since there are $m+1\choose i$ elements of rank $i$ in the semilattice
when $0\leq i \leq m-1$
(the joins of all $i$-subsets of $\{Q_0,\ldots,Q_m\}$).

Now the last sum is the negative of all terms of the binomial expansion of
$(1-1)^{m+1}$ (which is~$0$) except for the terms with $i=m$ and $i=m+1$. So
\[\mu(E,U)=(-1)^m{m+1\choose m}+(-1)^{m+1} = (-1)^mm,\]
as required. \qed
\end{proof}

\section{The spectrum of the diagonal graph}
\label{s:spectrum}

Put $\Omega=G^m$. For $0\leqslant i\leqslant m$, 
let $A_i$ be the adjacency matrix of the graph~$\Gamma_i$ on the
vertex set $\Omega$
in which two vertices are joined if they lie in the same part of $Q_i$.
This graph is the disjoint union of $q^{m-1}$ copies of the complete graph
of order~$q$; so its spectrum is $(q-1)$ with multiplicity $q^{m-1}$ and
$-1$ with multiplicity $q^{m-1}(q-1)$. Since the matrices $A_i$ commute
pairwise and their sum is the adjacency matrix $A$ of the diagonal graph
$\Gamma_D(G,m)$, we see that the set of eigenvalues of $A$ is contained in
\[\{-(m+1)+kq:0\leqslant k\leqslant m+1\}.\]

Computing the multiplicities is a little more difficult. We use the technique
from experimental design for finding stratum dimensions in Tjur block
structures: see~\cite{goodbook1,goodbook2,goodbook}.

Let $V=\mathbb{R}^\Omega$ be the vector space of real functions on $\Omega$.
For each element $S$ in the diagonal semilattice, let $V_S$ be the subspace
of functions which are constant on the parts of $S$. If $\rho(S)=k$, then
$\dim(V_S)=q^{m-k}$. For convenience we will denote $V_{Q_i}$ by $V_i$.

Any vector in $V_i$ is an eigenvector for $A_i$ with eigenvalue $q-1$, while
any vector orthogonal to $V_i$ is an eigenvector for $A_i$ with
eigenvalue~$-1$.
Hence, if $v\in V_S$ but $v$ is orthogonal to $V_i$ whenever $Q_i$ is not
below $S$, then $v$ is an eigenvector of $A$ with eigenvalue $-(m+1)+\rho(S)q$.

Clearly $V_U$ consists of constant functions and has dimension $1$, associated
with the eigenvalue $(m+1)(q-1)$.

For $S\ne U$, let
\[W_S=V_S\cap\left(\sum_{S\prec T}V_T\right)^\perp.\]
We calculate $n_S=\dim(W_S)$ by M\"obius inversion. We have
\[V_S=\bigoplus_{S\preccurlyeq T}W_T,\]
so
\[q^{m-\rho(S)}=\sum_{S\preccurlyeq T}n_T.\]
M\"obius inversion gives
\[n_S = \sum_{S\preccurlyeq T}\mu(S,T)q^{m-\rho(T)}.\]
Suppose that $m-\rho(S)=s$.
Using the fact that the interval $[S,U]$ is isomorphic
to the diagonal semilattice of dimension $m-\rho(S)=s$,
as well as the results of Theorem~\ref{th:Mob}, we have
\begin{eqnarray*}
 n_S   &=& \sum_{i=0}^{s-1}{s+1\choose i}(-1)^iq^{s-i}+(-1)^ss\\
    &=& q^{-1}((q-1)^{s+1}-(s+1)(-1)^sq-(-1)^{s+1}+(-1)^ssq)\\
    &=& q^{-1}((q-1)^{s+1}+(-1)^{s+1}(q-1))\\
    &=& q^{-1}(q-1)((q-1)^s-(-1)^s).
\end{eqnarray*}

Since there are $m+1\choose k$ elements of the diagonal semilattice with
rank $k$, for $0\leqslant k\leqslant m-1$, we conclude:

\begin{theorem}
If $G$ is a group of order~$q$, then the adjacency matrix of the diagonal graph
$\Gamma_D(G,m)$ has eigenvalue $-(m+1)+kq$ with multiplicity
\[{m+1\choose k}q^{-1}(q-1)((q-1)^{m-k}-(-1)^{m-k}),\]
for $k\in\{0,\ldots,m-1\}$, and eigenvalue $(m+1)(q-1)$ with multiplicity $1$.
\qed
\end{theorem}

\paragraph{Remark} The formula for $n_S$ differs only in a factor of $q$ from
the chromatic polynomial of the $(s+2)$-cycle graph (which can be obtained by
a similar M\"obius inversion over the poset of all subsets of the edge set
of the cycle graph). 

\paragraph{Remark} This formula agrees with the known formula for the cases
$m=2$ (Latin square graphs, with eigenvalues $3(q-1)$, $q-3$, $-3$ having
multiplicities $1$, $3(q-1)$, $(q-1)(q-2)$), and $q=2$ (the folded cube, where
the eigenvalues $-(m+1)+2k$ occur only if $m-k$ is odd, since
$(q-1)^{m-k}-(-1)^{m-k}=0$ when $q=2$ and $m-k$ is even).

\section{Clique number and chromatic number}
\label{s:chromatic}

Small diagonal graphs (of dimension~$2$ over groups of orders at most~$4$)
are somewhat exceptional. It is easy to see that
\begin{itemize}\itemsep0pt
\item $\Gamma_D(C_2,2)$ is the complete graph $K_4$;
\item $\Gamma_D(C_3,2)$ is the complete multipartite graph $K_{3,3,3}$ with
three parts of size~$3$;
\item $\Gamma_D(V_4,2)$ (where $V_4=C_2\times C_2$) is the complement of the
$4\times 4$ square lattice graph $L_2(4)$ (the line graph of the complete
bipartite graph $K_{4,4}$);
\item $\Gamma_D(C_4,2)$ is the complement of the \emph{Shrikhande graph},
the exceptional strongly regular graph having the same parameters as 
$L_2(4)$, discovered by Shrikhande~\cite{shrikhande}.
\end{itemize}
With these exceptions, the clique number of $\Gamma_D(G,m)$ is equal to the
order of~$G$, and the cliques of maximum size are precisely the parts of the 
partitions $Q_0,\ldots,Q_m$ (the minimal partitions in the diagonal
semilattice $\mathcal{D}(G,m)$). In fact, if $m>2$, then every clique of size
greater than $1$ is of this form; that is, every edge is contained in a unique
maximal clique.

An important problem which we have not been able to answer completely is
to find the chromatic number of $\Gamma_D(G,m)$. Recall that the
\emph{chromatic number} $\chi(\Gamma)$ of a graph~$\Gamma$ is the smallest
number of colours needed to colour the vertices of~$\Gamma$ so that adjacent
vertices have different colours.

We approach this problem using the
notion of graph homomorphism. Let $\Gamma$ and $\Delta$ be graphs. A map
$F:V(\Gamma)\to V(\Delta)$ is a \emph{homomorphism} if it maps edges to
edges. (Non-edges may map to non-edges, or to edges, or collapse to single
vertices.) We note two important properties of graph homomorphisms:
\begin{itemize}\itemsep0pt
\item A homomorphism $\Gamma\to K_r$ (where $K_r$ is the complete graph on
$r$ vertices) is a proper vertex-colouring of $\Gamma$ with $r$ colours
(since adjacent vertices of $\Gamma$ must map to distinct vertices of $K_r$).
\item If $F$ is a homomorphism from $\Gamma$ to $\Delta$, then 
$\chi(\Gamma)\leqslant\chi(\Delta)$. For, if we give to a vertex
$v\in V(\Gamma)$ the colour of $F(v)$ in the colouring of $\Delta$, we have a 
proper colouring of $\Gamma$.
\end{itemize}

\begin{prop}\label{p:hom}
For $m>2$, the map
\[(g_1,g_2,g_3,g_4,\ldots,g_m)\mapsto(g_1g_2^{-1}g_3,g_4,\ldots,g_m)\]
is a homomorphism from $\Gamma_D(G,m)$ to $\Gamma_D(G,m-2)$.
\end{prop}

\begin{proof}
This is routine checking. Take an edge in $\Gamma_D(G,m)$, joining 
$(g_1,\ldots,g_m)$ to $(h_1,\ldots,h_m)$. There are two
possibilities:
\begin{itemize}\itemsep0pt
\item There is some $i$ such that $g_i\ne h_i$ and $g_j=h_j$ for $j\ne i$.
If $i\leqslant 3$, 
then $g_1g_2^{-1}g_3\ne h_1h_2^{-1}h_3$ and $g_j=h_j$ for $j>3$; if $i>3$
then $g_1g_2^{-1}g_3=h_1h_2^{-1}h_3$ and $g_j=h_j$ for $j>3$ and $j\ne i$. In
either case, the images are adjacent in $\Gamma_D(G,m-2)$.
\item There is some $x\in G$ such that $h_i=xg_i$ for $1\leqslant i\leqslant m$.
Then we have
\[h_1h_2^{-1}h_3=xg_1g_2^{-1}x^{-1}xg_3=x(g_1g_2^{-1}g_3)\]
and $h_i=xg_i$ for $i>4$. \qed
\end{itemize}
\end{proof}

\begin{prop}\label{p:hp}
The chromatic number of $\Gamma_D(G,2)$ is at least $|G|$, with equality
if and only if either $G$ has odd order or the Sylow $2$-subgroups of $G$
are non-cyclic.
\end{prop}

\begin{proof}
A \emph{complete mapping} on a group $G$ is a bijection $\phi:G\to G$ such
that the map $\psi$ given by $\psi(g)=g\phi(g)$ for $g\in G$ is also a 
bijection. It is well-known (see~\cite{paige,bccsz}) that the following are
equivalent:
\begin{itemize}\itemsep0pt
\item $G$ has a complete mapping;
\item the Cayley table of $G$ has a transversal;
\item the Cayley table of $G$ has an orthogonal mate;
\item the Latin square graph of the Cayley table of $G$ has chromatic number
equal to $|G|$.
\end{itemize}
The \emph{Hall--Paige conjecture}~\cite{hp} asserts that a group $G$ has a
complete mapping if and only if either $G$ has odd order or the Sylow
$2$-subgroups of $G$ are not non-trivial cyclic groups. The conjecture was
proved in 2009 by Wilcox, Evans and Bray (see~\cite{wilcox,evans,bccsz}).
\qed
\end{proof}

From these two results, our main theorem follows:

\begin{theorem}
If $m$ is odd, or if $|G|$ is odd, or if the Sylow $2$-subgroups of $G$ are
non-cyclic, then $\chi(\Gamma_D(G,m)=m$. Otherwise,
$\chi(\Gamma_D(G,m)\leqslant\chi(\Gamma_D(G,2)$.
\end{theorem}

\begin{proof}
Assume that $m>2$.
If $m$ is odd, then Proposition~\ref{p:hom} shows that
$\chi(\Gamma_D(G,m))\leqslant\chi(\Gamma_D(G,1)$; but $\Gamma_D(G,1)$ is a
complete graph of order $|G|$. Also, $\Gamma_D(G,m)$ contains a clique of size
$|G|$; so its chromatic number is $|G|$.

If $m$ is even, the result follows from Propositions~\ref{p:hom} and
\ref{p:hp}. \qed
\end{proof}

We end this section with two remarks. First, the chromatic number of 
$\Gamma_D(C_2,m)$ (the folded cube) for $m$ even is known to be $4$ 
(see~\cite{payan}). Second, it is conjectured in~\cite{ghm} that, if $G$ has
non-trivial cyclic Sylow $2$-subgroups, then the chromatic
number of $\Gamma_D(G,2)$ (the Latin square graph of the Cayley table of $G$)
is $|G|+2$. If so, then this would be an upper bound for $\chi(\Gamma_D(G,m))$
for even $m\geqslant2$, and we further conjecture that this bound is attained
for all even~$m$.

\section{Other properties}

In the final section we give a few more properties of the diagonal graph.

\subsection{Basic properties}

We list here some properties established in~\cite{diag}.

\begin{prop}
Assume that $m\geqslant2$ and $|G|\geqslant2$, and let $\Gamma$ denote the
diagonal graph $\Gamma_D(G,m)$.
\begin{enumerate}\itemsep0pt
\item There are $|G|^m$ vertices, and the valency is $(m+1)(|G|-1)$.
\item Except for the case $m=|G|=2$, the clique number is $|G|$, and the
clique cover number is $|G|^{m-1}$.
\item $\Gamma_D(G_1,m_1)$ is isomorphic to $\Gamma_D(G_2, m_2)$ if and only if
$m_1=m_2$ and $G_1\cong G_2$.
\item The diameter of $\Gamma$ is $m+1-\lceil(m+1)/|G|\rceil$, which is at
most $m$, with equality if and only if $|G|>m+1$.
\end{enumerate}
\end{prop}

\subsection{Symmetry}

\begin{prop}
Assume that $m\geqslant2$ and $|G|\geqslant2$, and let $\Gamma$ denote the
diagonal graph $\Gamma_D(G,m)$.
\begin{enumerate}\itemsep0pt
\item $\Gamma$ is a Cayley graph for the group $G^m$, and hence is
vertex-transitive.
\item $\Gamma$ is edge-transitive if and only if $G$ is elementary abelian.
\item $\Gamma$ is vertex-primitive if and only if
  \begin{itemize}
  \item $G$ is characteristically simple (that is, a direct product
of simple groups), and
  \item if $G$ is an elementary abelian $p$-group then $p\nmid m+1$.
  \end{itemize}
\item If $m>2$ or $|G|>4$, then $\Aut(\Gamma)$ is transitive on the set of
cliques of maximal size in $\Gamma$.
\end{enumerate}
\end{prop}

\begin{proof}
(a) Recall that $\Aut(\Gamma)$ is the diagonal group $D(G,m)$, and also the
list of generators of the diagonal group given in the first section. The
generators of the first type in the list form a group isomorphic to $G^m$,
acting regularly on the vertex set of $\Gamma$; so $\Gamma$ is a Cayley graph.

We remark here that the generators of the fourth and fifth type generate a
group isomorphic to $\Sym(m+1)$ normalising the group generated by those of
the first two types, and acting as the symmetric group on the $m+1$ partitions
$Q_0,\ldots,Q_m$. (The fourth type induce $\Sym(m)$ on $Q_1,\ldots,Q_m$, while
the fifth type interchanges $Q_0$ and $Q_1$ and fixes the
others.) If $G$ is non-abelian, then the group $G^m$ generated by elements
of the first type has $m+1$ images under $\Sym(m+1)$; so there are $m+1$
subgroups of $\Aut(\Gamma)$ acting regularly, and so $\Gamma$ is a Cayley
graph for $m+1$ distinct, though isomorphic, groups.
On the other hand, if $G$ is abelian then elements of the second type are
already contained in the group $G^m$ generated by elements of the first type,
and so the images of this subgroup under $\Sym(m+1)$ are all equal.

\medskip

(b) The stabiliser of a vertex has structure $\Aut(G)\times\Sym(m+1)$. The
symmetric group permutes the $m+1$ cliques of size $|G|$ containing the 
fixed vertex; each clique is bijective with $G$, the fixed vertex corresponding
to the identity, and $\Aut(G)$ acts on it in the natural way. Thus $\Gamma$
is edge-transitive if and only if $\Aut(G)$ is transitive on the non-identity
elements of $G$, which occurs if and only if $G$ is elementary abelian.

\medskip

(c) This is proved in \cite[Theorem 1.6]{diag}.

\medskip

(d) This follows from the above remarks: each such clique is a part of a 
partition $Q_i$; now $\Sym(m+1)$ permutes the partitions transitively, and
$G^m$ fixes each partition and permutes transitively the parts of each. \qed
\end{proof}

Recall the definition of distance-regularity from \cite{bcn}.

\begin{prop}
The graph $\Gamma_D(G,m)$ is distance-regular if and only if either $m=2$
or $|G|=2$.
\end{prop}

\begin{proof}
If $m=2$, then $\Gamma_D(G,m)$ is a Latin square graph, which is strongly
regular; if $G=C_2$, then $\Gamma_D(G,m)$ is a folded cube, which is
distance-transitive. In either case the graph is distance-regular.

So suppose that $m>2$ and $|G|>2$.

First consider the case where $m$ is odd, say $m=2k+1$. We argue as in
Example 2 given earlier. Take two non-identity
elements $g,h\in G$. Let $a$ and $b$ be the elements of $G^m$ given by
$a=(g,\ldots,g,1,\ldots,1)$ and $b=(h,g,\ldots,g,1,\ldots,1)$, where in each
case the number of non-identity elements is $k+1$. Each of these elements has
distance $k+1$ from the identity element of $G^m$. In the case of $b$, there
are $k+1$ neighbours at distance $k$ from the identity, obtained by changing
one of the non-identity entries in $h$ to the identity. In the case of $a$,
there are $k+2$ such elements, since additionally we may multiply on the
left by $g^{-1}$ to get $(1,\ldots,1,g^{-1},\ldots,g^{-1})$ with $k$
non-identity elements.

Now suppose that $m$ is even, say $m=2k$ with $k>1$. Again take $g,h$ to be
non-identity elements in $G$, and let
$a=(g,\ldots,g,1,\ldots,1)$ and $b=(h,g,\ldots,g,1,\ldots,1)$, where in each
case there are $k$ non-identity elements. These vertices are at distance
$k$ from the identity; we count neighbours which are also at distance $k$
from the identity. In the case of $b$, there are $k(|G|-2)$ of these, since
we may change any of the first $k$ elements to a different non-identity
element of $G$. In the case of $a$, there is one extra such vertex, namely
$(1,\ldots,1,g^{-1},\ldots,g^{-1})$. 

In both cases the graph fails to be distance-regular.
\qed
\end{proof}

\end{document}